\documentclass[12pt,a4paper]{amsart}
\usepackage{amssymb, amsmath}
\usepackage{xcolor,hyperref} 
\hypersetup{
    colorlinks=true,
    citecolor=blue,
    linkcolor=red,
    urlcolor=orange,
}
\newtheorem{thm}{Theorem}
\newtheorem*{prop}{Proposition}
\newtheorem{conj}{Conjecture}
\newtheorem*{prob}{Problem}
\newtheorem{lemma}{Lemma}

\def\FF{\mathcal F}
\def\GG{\mathcal G}
\def\HH{\mathcal H}
\def\QQ{\mathcal R}
\def\RR{\mathcal Q}
\def\TT{\mathcal T}
\title{Strong stability of 3-wise $t$-intersecting families}
\author{Norihide Tokushige}
\address{University of the Ryukyus\\College of Education\\
1 Senbaru Nishihara\\Okinawa\\903-0213 (JAPAN)\\
ORCID: 0000-0002-9487-7545}
\email{hide@edu.u-ryukyu.ac.jp}

\begin{document}
\begin{abstract}
Let $\GG$ be a family of subsets of an $n$-element set. 
The family $\GG$ is called $3$-wise $t$-intersecting if the intersection of
any three subsets in $\GG$ is of size at least $t$.
For a real number $p\in(0,1)$ we define the measure of the family 
by the sum of $p^{|G|}(1-p)^{n-|G|}$ over all $G\in\GG$. 
For example, if $\GG$ consists of all subsets containing a fixed $t$-element
set, then it is a $3$-wise $t$-intersecting family with the measure $p^t$.

Let $0<p\leq 2/(\sqrt{4t+9}-1)$, $\delta>0$, and let $\GG$ be a $3$-wise 
$t$-intersecting family. It is known that the measure of $\GG$ is at most $p^t$. 
Suppose, moreover, that $\GG$ has the measure at least $(\frac12+\delta)p^t$.
We show that, by choosing $t$ sufficiently large depending on $\delta$, 
the structure of $\GG$ is one of (i) and (ii): (i) every subset in $\GG$ contains
a fixed $t$-element set, (ii) every subset in $\GG$ contains
at least $t+2$ elements from a fixed $(t+3)$-element set.
\end{abstract}
\maketitle
\section{Introduction}
Let $X$ be a finite set, and let $2^X$ denote the power set of $X$.
Let $\GG\subset 2^X$ be a family of subsets of $X$\footnote{In this paper
we write $A\subset B$ to mean that $A$ is a subset of $B$,
including the case $A=B$.}. 
For positive integers $r$ and $t$, we say that 
a family $\GG$ is \emph{$r$-wise $t$-intersecting} if 
$|G_1\cap G_2\cap\cdots\cap G_r|\geq t$ for all $G_1,G_2,\ldots, G_r\in\GG$.
We are interested in `large' $r$-wise $t$-intersecting families with respect to
the measure introduced below.
Let $0<p<1$ be a real number and let $q:=1-p$. For a family $\GG\subset 2^X$
we define the $p$-biased measure of $\GG$ on $X$ by
\[
\mu_p(\GG:X):=\sum_{G\in\GG}p^{|G|}q^{|X|-|G|}.
\]
Note that $\mu_p(\GG:X)$ equals the probability ${\mathbb P}(X_p\in\GG)$ where
$X_p$ is a random subset of $X$ obtained by retaining each element of $X$
with probability $p$, independently.
For the case $X=[n]:=\{1,2,\ldots, n\}$
we just write $\mu_p(\GG)$ to mean $\mu_p(\GG:[n])$.

For example, if $r=2$ and $p$ is small, then we have the following
result, which is a special case of the $p$-biased version (see, e.g., Theorem~12.4\footnote{
The full statement of the $p$-biased version of the complete intersection theorem
was perhaps first appeared in \cite{DS,Tuvsw}, 
but the case for $p=\frac1m$, where $m$ is a positive integer, was already obtained 
in \cite{AK-p,FT1999}, see also \cite{BE, Filmus,Friedgut}.
} in \cite{EPFS})
of the complete intersection theorem \cite{AK} due to Ahlswede and Khachatrian.
We say that two families $\GG,\HH\subset 2^{[n]}$ are \emph{isomorphic},
denoted by $\GG\cong\HH$, if there is a permutation $\tau$ on $[n]$
such that $\HH=\{\{\tau(x):x\in G\}:G\in\GG\}$.
\begin{thm}\label{thm1}
Let $0<p<\frac1{t+1}$. If $\GG\subset 2^{[n]}$ is a $2$-wise 
$t$-intersecting family, then $\mu_p(\GG)\leq p^t$. 
Moreover, equality holds if and only if
$\GG$ is isomorphic to $\GG^*:=\{G\subset[n]:[t]\subset G\}$.
\end{thm}
Theorem~\ref{thm1} tells us that if $p<\frac1{t+1}$ then $\GG^*$ is the unique 2-wise
$t$-intersecting family, up to isomorphism, with $p$-biased measure $p^t$.
This property is stable in the sense that if a 2-wise $t$-intersecting
family $\GG$ has the measure close to $\mu_p(\GG^*)=p^t$, then $\GG$
is `close to' $\GG^*$. Here, $\GG$ is close to $\GG^*$ means that
$\mu_p(\GG\triangle\GG^*)$ is small. More precisely, Friedgut proved
the following.
\begin{thm}[\cite{Friedgut}]\label{thm2}
Let $0<p<\frac1{t+1}$ and $\epsilon>0$. There is $C=C(p)$ such that
if $\GG\subset 2^{[n]}$ is a $2$-wise $t$-intersecting family
with $\mu_p(\GG)>p^t-\epsilon$, then $\mu_p(\GG\triangle\HH)\leq C\epsilon$
for some $\HH\cong\GG^*$.
\end{thm}
We note that a 2-wise $t$-intersecting family which is close to
$\GG^*$ is \emph{not necessarily} a subfamily of $\GG^*$.
For example, let $\GG':=(\GG^*\setminus\{[t]\})\cup\{[n]\setminus\{1\}\}$.
Then $\GG'$ is a 2-wise $t$-intersecting family with
$\mu_p(\GG'\triangle\GG^*)=p^tq^{n-t}+p^{n-1}q\to 0$ as $n\to\infty$ for 
fixed $p,t$. Thus $\GG'$ has the measure close to $\GG^*$,
but it is not a subfamily of $\GG^*$.

Ellis, Keller, and Lifshitz obtained a similar, much stronger stability 
result for all $\zeta\leq p<\frac12-\zeta$. We include the result for 
the case $p<\frac1{t+1}$ below. (See also \cite{EKL19} for the corresponding
result for the $k$-uniform families.)
\begin{thm}[\cite{EKL18}]\label{thm3}
For every $\zeta>0$ there is $C=C(\zeta)$ such that the following holds.
Let $\zeta\leq p<\frac1{t+1}$ and $\epsilon>0$. If $\GG\subset 2^{[n]}$
is a $2$-wise $t$-intersecting family with $\mu_p(\GG)>p^t-\epsilon$,
then $\mu_p(\GG\setminus\HH)\leq C\epsilon^{\log_{1-p}p}$ for some
$\HH\cong\GG^*$.
\end{thm}

This note addresses the same problem for the 3-wise $t$-intersecting families
and shows that the situation in this case is different from the 2-wise 
$t$-intersecting case. For this, we introduce
$r$-wise $t$-intersecting families $\FF_i^t(r)$ for 
$i=0,1,\ldots,\lfloor\frac{n-t}r\rfloor$ by
\[
 \FF_i^t(r):=\{F\subset[n]:|F\cap[t+ri]|\geq t+(r-1)i\}.
\]
Note that $\FF_0^t(r)=\GG^*$.
From now on we mainly consider 3-wise $t$-intersecting families,
and we write $\FF_i^t$ to mean $\FF_i^t(3)$.
By comparing $\mu_p(\FF_0^t)=p^t$ and $\mu_p(\FF_1^t)=p^{t+3}+(t+3)p^{t+2}q$
we see that $\mu_p(\FF_1^t)<\mu_p(\FF_0^t)$ if and only if $0<p<p_0$, where
\begin{align}\label{p0}
p_0=p_0(t):=\frac2{\sqrt{4t+9}-1}. 
\end{align}
Moreover, $\mu_p(\FF_1^t)=\mu_p(\FF_0^t)$ if and only if $p=p_0$.
In \cite{T2023EJC} the corresponding result to Theorem~\ref{thm1}
is obtained as follows.
\begin{thm}[\cite{T2023EJC}]
Let $t\geq 15$ and $0<p<p_0$.
Let $\GG\subset 2^{[n]}$ be a $3$-wise $t$-intersecting family.
Then $\mu_p(\GG)\leq p^t$. Moreover, equality holds if and only if 
$\GG\cong\FF_0^t$.
\end{thm}
In \cite{T2023EJC} it is conjectured that the same result holds for all 
$t\geq 2$ (the case $t=1$ is different).
Also, in \cite{T2023EJC}, a stability result for 3-wise intersecting 
families is obtained, which is similar to Theorem~\ref{thm2}, 
but only valid for shifted families. (See subsection~\ref{shifting} for the
definition of shifted families.)
In fact, however, a much stronger stability was expected.
\begin{conj}[\cite{T2023EJC}]\label{conj1}
Let $t\geq 2$ and $0<p<p_0$.
Then there exists $\epsilon=\epsilon(p,t)>0$ satisfying the following 
statement: if $\GG\subset 2^{[n]}$ is a $3$-wise $t$-intersecting family with
$\mu_p(\GG)>p^t-\epsilon$, then $\GG\subset\FF$ for some $\FF\cong\FF_0^t$.
\end{conj}
This conjecture turned out to be false if $t=2,3$. To see this, define
a 3-wise $t$-intersecting family $\RR$ by
\[
\RR:=\{[t]\cup A:A\subset[t+1,n],\,|A|\geq\tfrac n2\} 
\cup\{B\cup[t+1,n]:B\subset 2^{[t]}\},
\]
where $[i,j]:=[j]\setminus[i-1]$ for positive integers $i,j$ with $i<j$.
If $t=2$ or $3$, then $\mu_p(\RR)\to p^t$ as $n\to\infty$ provided 
$\frac12<p\leq p_0(t)$, and so the conjecture fails.
But if $t\geq 4$ then $p_0(t)\leq\frac12$ and $\mu_p(\RR)\ll p^t$.
If Conjecture~\ref{conj1} is still true for $t\geq 4$, then there is a big difference 
between the 2-wise $t$-intersecting case and the $3$-wise $t$-intersecting case.
That is, in the latter case a family with the measure close to $p^t$ is 
\emph{necessarily} a subfamily of the optimal family $\FF_0^t=\GG^*$.

In this note we verify Conjecture~\ref{conj1} when $t$ is large.
To motivate our result we compare the measure of $\FF_0^t$ and $\FF_2^t$.
In subsection~\ref{subsec:1/2}, we will show 
\begin{align}
\max_{0<p\leq p_0}\frac{\mu_p(\FF_2^t)}{\mu_p(\FF_0^t)}
\to\frac12 \label{eq:1/2}
\end{align}
as $t\to\infty$, via direct computation.
Now we can state our main result.

\begin{thm}\label{thm5}
For every $\delta>0$ there exists $t_0$ such that 
for all $t$ and $p$ with $t>t_0$ and $0<p\leq p_0(t)$, the following holds. 
If $\GG\subset 2^{[n]}$ is a $3$-wise $t$-intersecting family with
$\mu_p(\GG)\geq(\frac12+\delta) p^t$, then $\GG\subset\FF$ for some $\FF$ 
such that $\FF\cong\FF_0^t$ or $\FF\cong\FF_1^t$.
\end{thm}
This result roughly tells us that if a 3-wise $t$-intersecting family 
has measure larger than $\frac12p^t$ and also larger than 
$\mu_p(\FF_1^t)$, then it is necessarily a subfamily of $\FF_0^t$.
Thus the stability of 3-wise $t$-intersecting families is stronger than
that of 2-wise $t$-intersecting families, cf.~Theorems~\ref{thm2} 
and \ref{thm3}.

The condition $\mu_p(\GG)\geq(\frac12+\delta)p^t$ in Theorem~\ref{thm5}
is almost tight, because if $p$ is close to $p_0$ then $\mu_p(\FF_2^t)$ is 
close to $\frac12p^t$, see subsection~\ref{subsec:1/2}. 
We cannot replace the condition with 
$\mu_p(\GG)>\mu_p(\FF_1^t)$ to imply that $\GG$ is a subfamily of 
$\FF_0^t$, because $\FF_1^t$ is not necessarily the 
second largest choice if $p$ is small. 
To see this, consider
the family
\begin{align}\label{eq:Q}
\QQ:=\big\{[t+1]\cup A:\emptyset\neq A\subset[t+2,n]\big\}\cup
\big\{[n]\setminus\{i\}:i\in[n]\big\}.
\end{align}
This is a 3-wise $t$-intersecting family which
is not a subfamily of $\FF_0^t$, $\FF_1^t$, or $\FF_2^t$.
But the measure of $\QQ$ is approximately $p^{t+1}$ if $n$ is sufficiently 
large, and $p^{t+1}$ is larger than $\mu_p(\FF_1^t)$ for $p<\frac1{t+2}$.

We use a shifting operation to prove Theorem~\ref{thm5}. 
Usually such results are valid only for shifted families. 
One novelty of our result is that it avoids the condition that the families 
are shifted. This can be accomplished by careful analysis of the shifting 
process, see Lemma~\ref{lemma:matome}. Using this lemma we can focus on
shifted 3-wise $t$-intersecting families not included $\FF_0^t$, $\FF_1^t$,
or $\FF_2^t$. Then, in section~\ref{sec:proof}, we divide such families
into subfamilies with additional intersecting properties 
(Lemma~\ref{lemmaMIFR}), and bound the measure of each subfamily
to conclude Theorem~\ref{thm5}.

\section{Preliminaries}
\subsection{Comparing the measure of $\FF_0^t$ and $\FF_2^t$}\label{subsec:1/2}
We prove \eqref{eq:1/2}. Let $q_0=1-p_0$. 
Since $p^2 q$ is increasing in $p$ for $0<p\leq p_0(t)$, we have
$p^4 q^2=(p^2 q)^2\leq (p_0^2q_0)^2$, 
$p^5 q=(p^2 q)p^3\leq (p_0^2 q_0)p_0^3$.
Thus we have
\begin{align*}
\max_{0<p\leq p_0}\frac{\mu_p(\FF_2^t)}{\mu_p(\FF_0^t)}
&=\max_{0<p\leq p_0}\left(
\binom{t+6}2p^4q^2+\binom{t+6}1p^5q+\binom{t+6}0p^6\right)\\
&=\binom{t+6}2p_0^4q_0^2+\binom{t+6}1p_0^5q_0+\binom{t+6}0p_0^6\\
&=\frac{16(2 t^3-3 t^2\sqrt{4 t+9}+31t^2-31 t\sqrt{4 t+9}+153t-78 \sqrt{4 t+9}+238)}{\left(\sqrt{4 t+9}-1\right)^6},
\end{align*}
and the RHS goes to $1/2$ as $t\to\infty$.

\subsection{Shifting and related results}\label{shifting}
For $1\leq i<j\leq n$ we define the shifting operation 
$\sigma_{i,j}:2^{[n]}\to 2^{[n]}$ by 
\[
\sigma_{i,j}(\GG):=\{s_{i,j}(G):G\in\GG\}, 
\]
where
\[
s_{i,j}(G):=\begin{cases}
(G\setminus\{j\})\cup\{i\}&\text{if }G\cap\{i,j\}=\{j\}
\text{ and }(G\setminus\{j\})\cup\{i\}\not\in\GG,
\\
G & \text{otherwise}.	   
\end{cases}
\]
By definition $\mu_p(\GG)=\mu_p(\sigma_{i,j}(\GG))$.
We say that $\GG$ is \emph{shifted} if $\GG$ is invariant under any shifting 
operation, in other words, if $G\in\GG$ then $s_{i,j}(G)\in\GG$ for all 
$1\leq i<j\leq n$. If $\GG$ is not shifted then 
$\sum_{G\in\GG}\sum_{g\in G}g>\sum_{G'\in\sigma_{i,j}(\GG)}\sum_{g'\in G'}g'$
for some $i,j$ with $i<j$, and so starting from $\GG$ we get a shifted 
$\HH$ by applying shifting operations a finite number of times.
We say that this $\HH$ is \emph{obtained from $\GG$ by shifting}.
From the definition of the shifting operation, the following simple fact 
follows.
\begin{lemma}
Let $1\leq i<j\leq n$, and 
let $\GG\subset 2^{[n]}$ be an $r$-wise $t$-intersecting family.
Then the family $\sigma_{i,j}(\GG)$ is also $r$-wise $t$-intersecting.
In particular, there is a shifted $r$-wise $t$-intersecting family $\HH$
obtained from $\GG$ by shifting such that $\mu_p(\HH)=\mu_p(\GG)$.
\end{lemma}

We say that a family $\GG\subset 2^{[n]}$ is \emph{size maximal} 
with respect to the $r$-wise $t$-intersecting property, 
or \emph{$(r,t)$-maximal} for short, 
if $\GG$ is $r$-wise $t$-intersecting, but $\GG\cup\{F\}$ is not $r$-wise 
$t$-intersecting for every $F\in 2^{[n]}\setminus\GG$.
Note that if $\GG$ is size maximal, then it is inclusion maximal, that is,
if $G\in\GG$ and $G\subset H$ then $H\in\GG$.

\begin{lemma}\label{lemma:F0}
Let $\GG\subset 2^{[n]}$ be $(3,t)$-maximal,
and let $\HH=\sigma_{i,j}(\GG)$.
If $\HH\subset\FF_0^t$ then $\GG\subset\FF$ for some $\FF$ isomorphic to 
$\FF_0^t$, and thus $\GG\cong\FF_0^t$ and $\HH=\FF_0^t$.
\end{lemma}

\begin{proof}
If $i,j\in[t]$ or $i,j\in[n]\setminus[t]$, then the result easily follows
since $[t]\subset G$ if and only if $[t]\subset s_{i,j}(G)$.
So we may assume that $i\in[t]$ and $j\in[n]\setminus[t]$, say,
$i=t$, $j=t+1$.
Since $\HH\subset\FF_0^t$, every $G\in\GG$
satisfies that $[t]\subset s_{i,j}(G)$. Thus we have $[t-1]\subset G$
and $|G\cap\{i,j\}|\geq 1$. 

If $\GG$ contains a $t$-element set, then 
$\GG\cong\FF_0^t$ because $\GG$ is $(3,t)$-maximal.
So we assume that $\GG$ contains no $t$-element set, in particular,
$[t]\not\in\GG$. Then it follows from $(3,t)$-maximality that 
$\GG\cup\{[t]\}$ is not 3-wise $t$-intersecting. 
Thus there are $A,B\in\GG$ such that $|[t]\cap A\cap B|<t$, and
one of them, say, $A$ satisfies $A\cap\{i,j\}=\{j\}$. 
Since $\GG$ is inclusion maximal we may assume that $A=[n]\setminus\{i\}$.
Then $H:=[n]\setminus\{j\}\not\in\GG$, because otherwise
$s_{i,j}(A)=A\in\HH$ but $A\not\in\FF_0^t$.
Thus every subset of $H$ cannot be
in $\GG$, in other words, $j\in G$ for all $G\in\GG$.
Let $T=[t-1]\cup\{j\}$. Since $T$ is a $t$-element set, we have
$T\not\in\GG$, but $\GG\cup\{T\}$ is still 3-wise $t$-intersecting,
which contradicts the assumption that $\GG$ is size maximal.
\end{proof}

\begin{lemma}\label{lemma:F1}
Let $\GG\subset 2^{[n]}$ be a $(3,t)$-maximal family, 
and let $\HH=\sigma_{i,j}(\GG)$.
If $\HH\subset\FF_1^t$ then $\GG\subset\FF$ for some $\FF$ isomorphic to 
$\FF_1^t$, and thus $\GG\cong\FF_1^t$ and $\HH=\FF_1^t$.
\end{lemma}

\begin{proof}
We may assume that $i=t+3$, $j=t+4$.
For $l\in[t+2]$ let 
$G_l:=[n]\setminus\{l,i\}$ and $K_l:=[n]\setminus\{l,j\}$.
Note that $|\{G_l,K_l\}\cap\GG|\leq 1$. Indeed, if $\{G_l,K_l\}\subset\GG$ then
$s_{i,j}(G_l)=G_l\in\HH$, but $G_l\not\in\FF_1^t$, a contradiction.

It follows from $\sigma_{i,j}(\GG)\subset\FF_1^t$ that
$|[t+2]\setminus G|\leq 1$ for all $G\in\GG$. Thus we have
$|[t+2]\cap G\cap G'|\geq t$ for all $G,G'\in\GG$.
We also note that $|[t+4]\setminus G|\leq 2$ for all
$G\in\GG$, and if $G\in\GG$ satisfies $|[t+4]\setminus G|=2$ then $|G\cap\{i,j\}|\leq 1$.

Let $\tilde\GG:=\{G\in\GG:|[t+4]\setminus G|=2\}$. 
If $j\not\in G$ for all $G\in\tilde\GG$ then $\GG\subset\FF_1^t$. 
If $i\not\in G$ for all $G\in\tilde\GG$ then 
$\GG\subset\{G:|G\cap([t+2]\cup\{j\})|\geq t+2\}\cong\FF_1^t$. 
So we may assume that there exist some $l,m\in[t+2]$ such that
$G_l,K_m\in\tilde\GG$. Then $\{K_l,G_m\}\cap\GG=\emptyset$, and so $m\neq l$.
Since $G_m\not\in\GG$ there exist $A,B\in\GG$ such that 
$|A\cap B\cap G_m|\leq t-1$. Then we may assume that 
$[t+4]\setminus A=\{k,j\}$ for some $k\in[t+2]\setminus\{l,m\}$,
and so $K_k\in\GG$.

Since $K_l\not\in\GG$ there exist $C,D\in\GG$ such that $|C\cap D\cap K_l|\leq t-1$. 
We also have $|C\cap D\cap K_l\cap[t+2]|\geq t-1$. Thus we have
$|C\cap D\cap K_l\cap[t+2]|= t-1$ and $C\cap D\cap K_l\cap [i,n]=\emptyset$,
in particular, $i\not\in C\cap D$ and $C\cap D\cap [j+1,n]=\emptyset$.

As $i\not\in C\cap D$, suppose that $i\not\in C$.
Let $[t+2]\setminus C=\{c\}$ and $[t+2]\setminus D=\{d\}$.
Since one of $m$ and $k$ does not coincide with $d$, we may assume that $m\neq d$.
Since $C\subset G_c\in\GG$ we have $K_c\not\in\GG$, and so $c\neq m$. Thus $c,d$ and $m$
are all distinct, and we have $|C\cap D\cap K_m|=t-1$, a contradiction.
\end{proof}

By Lemmas~\ref{lemma:F0} and \ref{lemma:F1} we have the following.

\begin{lemma}\label{lemma:matome}
Let $\GG\subset 2^{[n]}$ be a $(3,t)$-maximal family, and let $\HH$ be 
a shifted family obtained from $\GG$ by shifting. 
If $\HH\subset\FF_i^t$ for $i=0$ or $i=1$, then $\GG\cong\FF_i^t$.
\end{lemma}

\subsection{Inequalities for the measure of 2-wise $t$-intersecting families}

\begin{lemma}\label{EKR}
Let $\GG\subset 2^{[n]}$ be a $2$-wise $t$-intersecting family.
\begin{itemize}
 \item[(i)] If $0<p\leq 1/2$ then $\mu_p(\GG)\leq(p/q)^t$.
 \item[(ii)] If $0<p\leq\frac1{t+1}$ then $\mu_p(\GG)\leq p^t$.
\end{itemize}
\end{lemma}

For the proof of (i), see Theorem 15.7 in \cite{EPFS}, which is
a consequence of the so-called random walk method.
In Theorem~\ref{thm1} we have $\mu_p(\GG)\leq p^t$ 
not only for $p<\frac1{t+1}$ but also for $p=\frac1{t+1}$,
and we have (ii), see e.g., \cite{Filmus,Tuvsw}.
It is also known that
if $p=\frac1{t+1}$ then $\mu(\GG)=p^t$ holds if and
only if $\GG\cong\GG^*=\FF_0^t(2)$ or $\GG\cong\FF_1^t(2)$,
but we do not use this fact.

\section{Proof of Theorem~\ref{thm5}}\label{sec:proof}
We prove the following stronger statement, which immediately implies the
theorem.
\begin{prop}
Let $t\geq 241$ and $0<p\leq p_0(t)$.
If $\GG\subset 2^{[n]}$ is a $3$-wise $t$-intersecting family, then
one of the following holds.
\begin{itemize}
\item[(i)] We have $\mu_p(\GG)<\frac12p^t$.
\item[(ii)] We have $\GG\subset\FF$ for some $\FF$ isomorphic to $\FF_0^t$
or $\FF_1^t$.
\item[(iii)] There exists a $3$-wise intersecting family $\tilde\GG$ 
with $\GG\subset\tilde\GG$ such that a shifted family $\HH$ obtained from 
$\tilde\GG$ by shifting satisfies $\HH\subset\FF_2^t$.
Moreover, for every $\delta>0$ we have
$\mu_p(\GG)<(\frac12+\delta)p^t$ provided $t>t_0(\delta)$.
\end{itemize}
\end{prop}

\begin{proof}
Let $t\geq 241$, $0<p\leq p_0$, and let $\GG\subset 2^{[n]}$ be a 3-wise
$t$-intersecting family. By adding subsets to $\GG$ if necessary, we get 
a $(3,t)$-maximal family $\tilde \GG$ such that $\GG\subset\tilde\GG$.
Let $\HH\subset 2^{[n]}$ be a shifted 3-wise $t$-intersecting family obtained 
from $\tilde\GG$ by shifting. 
If $\HH\subset\FF_0^t$ or $\FF_1^t$, then we have 
$\tilde\GG\cong\FF_0^t$ or $\tilde\GG\cong\FF_0^1$ by Lemma~\ref{lemma:matome}.
In this case we get (ii) because $\GG\subset\tilde\GG$.
If $\HH\subset\FF_2^t$, then we have (iii). (The second statement follows from
\eqref{eq:1/2}.) So we may assume that 
$\HH\not\subset\FF_0^t$, $\HH\not\subset\FF_1^t$ and $\HH\not\subset\FF_2^t$.
We will show that (i) holds. To this end it suffices to show that 
$\mu_p(\HH)<\frac12p^t$ because $\mu_p(\GG)\leq\mu_p(\tilde\GG)=\mu_p(\HH)$.

We borrow some ideas from \cite{T2007b}, which was inspired by 
\cite{Frankl1991}.
Let $s$ be such that $\HH$ is 2-wise $s$-intersecting, 
but it is not 2-wise $(s+1)$-intersecting. Since $\HH\not\subset\FF_0^t$ and
$\HH$ is 3-wise $t$-intersecting, we have $s\geq t+1$. Define $h=h(\HH)$ by
\[
 h:=\min\{i:|H\cap[t+i]|\geq t\text{ for all }H\in\HH\}.
\]
By definition there exists $H_0\in\HH$ such that $|H_0\cap[t+(h-1)]|=t-1$.
Since $\HH$ is shifted, we may assume that 
\begin{align}\label{H0}
H_0\cap[t+h-1]=[t-1].
\end{align}
For $A\subset[t+h-1]$ let 
\[
\TT(A):=\{H\cap[t+h,n]:H\in\HH,\,[t+h-1]\setminus H=A\}\subset 2^{[t+h,n]}. 
\]
If $H\in\HH$ and $H\setminus[t+h-1]\in\TT(A)$, then 
$H$ has a `hole' $A$ in $[t+h-1]$.
For the rightmost hole of size $i$, we define 
$\TT_i:=\TT([t+h-i,t+h-1])\subset 2^{[t+h,n]}$. 
Since $\HH$ is shifted, we have $\TT(A)\subset\TT_i$ for all 
$A\in\binom{[t+h-1]}i$. Thus we have
\begin{align}\label{eq:sum of G(Ti)}
\mu_p(\HH)
\leq\sum_{i=0}^h\binom {t+h-1}i p^{t+h-1-i}q^i\mu_p(\TT_i:[t+h,n]).
\end{align}

The following result is a special case of Claims 2, 3 and 4 in \cite{T2007b},
and we include a proof for convenience.

\begin{lemma}\label{lemmaMIFR}
We have $1\leq h\leq s-t$, and for $0\leq i\leq h$ the family
$\TT_i$ is $2$-wise $(2i+1)$-intersecting.
If $\HH\not\subset\FF_h^t$ then $\TT_h$ is $2$-wise $(2h+2)$-intersecting.
\end{lemma}

\begin{proof}
Since $\HH\not\subset\FF_0^t$ we have $h\geq 1$. By the definition of $s$ there
are $H_1,H_2\in\HH$ such that $|H_1\cap H_2|=s$, and we may assume that
$H_1\cap H_2=[s]$ because $\HH$ is shifted. Then, since $\HH$ is 3-wise
$t$-intersecting, it follows that $|H\cap[s]|\geq t$ for all $H\in\HH$.
By the definition of $h$ we have $s\geq t+h$, that is, $h\leq s-t$.

To show that $\TT_i$ is 2-wise $(2i+1)$-intersecting, suppose, to the 
contrary, that there exist $T_1,T_2\in\TT_i$ such that 
$T_1\cap T_2=I_1\cup I_2$ with $|I_1|=|I_2|=i$. 
Note that $I_1$ and $I_2$ are not necessarily disjoint.
Let $A=[t+h-i,t+h-1]$ be the hole for $\TT_i$, and 
for $j=1,2$ let $H_j:=([t+h-1]\setminus A)\cup T_j\in\HH$. Since $\HH$ is
shifted we have $H_j':=(H_j\setminus I_j)\cup A\in\HH$ and 
$H_1'\cap H_2'=[t+h-1]$. Then, by \eqref{H0}, we have $H_0\in\HH$ and
$H_0\cap H_1'\cap H_2'=[t-1]$, which contradicts that $\HH$ is 3-wise
$t$-intersecting.

Suppose that $\HH\not\subset\FF_h^t$. Assume for contradiction that 
there are $T_1,T_2\in\TT_h$ such that $|T_1\cap T_2|=2h+1$.
Since $\HH$ is shifted, we may assume that there are $H_1,H_2\in\HH$ such 
that $H_1\cap H_2=[t-1]\cup[t+h,t+3h]$, 
and $H_1':=(H_1\setminus[t+h,t+2h-1])\cup[t,t+h-1]\in\HH$.
Then $H_1'\cap H_2=[t+3h]\setminus[t,t+2h-1]$.
Since $\HH\not\subset\FF_h^t$ and $\HH$ is shifted, there exists
$H_3$ such that $H_3\cap[t+3h]=[t+2h-1]$.
But $H_1'\cap H_2\cap H_3=[t-1]$, 
which contradicts the assumption that $\HH$ is 3-wise $t$-intersecting.
Consequently $\TT_h$ is 2-wise $(2h+2)$-intersecting. 
\end{proof}

We will complete the proof of Proposition, and hence Theorem~\ref{thm5}, by showing
$\mu_p(\HH)<p^t/2$ in a handful of different cases (Lemmas~\ref{lemma:h=1}--\ref{lemma:h big}) based on the value of $h$.

\begin{lemma}\label{lemma:h=1}
If $h=1$ and $t\geq 28$ then $\mu_p(\HH)<\frac12p^t$.
\end{lemma}
\begin{proof}
It follows from Lemma~\ref{lemmaMIFR} that $\TT_0$ is 2-wise 1-intersecting, 
and $\TT_1$ is 2-wise 4-intersecting because $\HH\not\subset\FF_1^t$.
If $t\geq 28$ then $p_0\leq \frac15$, and it follows from (ii) of Lemma~\ref{EKR} that $\mu_p(\TT_0:[t+h,n])\leq p$ and
$\mu_p(\TT_1:[t+h,n])\leq p^4$ for $p\leq p_0$.
Then, by \eqref{eq:sum of G(Ti)}, we have
\begin{align*}
\frac{\mu_p(\HH)}{p^t}&\leq \mu_p(\TT_0:[t+1,n])+t\,\frac qp\,\mu_p(\TT_1[t+1,n])\\
&\leq p+tp^3q\leq p_0+tp_0^3q_0
=\frac2{\sqrt t}+O(1/t).
\end{align*}
Thus $\mu_p(\HH)/p^t\to 0$ as $t\to\infty$. 
Moreover, $p_0+tp_0^3q_0<\frac12$ for $t\geq 15$.
\end{proof}

\begin{lemma}
If $h=2$ and $t\geq 54$ then $\mu_p(\HH)<\frac12p^t$.
\end{lemma}
\begin{proof}
By Lemma~\ref{lemmaMIFR}, $\TT_i$ is 2-wise $(2i+1)$-intersecting for $i=0,1$.
Also $\TT_2$ is 2-wise 6-intersecting, where we use $\HH\not\subset\FF_2^t$.
To ensure from (ii) of Lemma~\ref{EKR}
that $\mu_p(\TT_2:[t+2,n])\leq p^6$, we need that $p\leq\frac17$,
which follows from $t\geq 54$, or equivalently $p_0\leq\frac17$.
We also have $\mu_p(\TT_0:[t+2,n])\leq p$ and $\mu_p(\TT_1:[t+2,n])\leq p^3$.
Using \eqref{eq:sum of G(Ti)} we have
\begin{align}
\frac{\mu_p(\HH)}{p^t}&\leq
\sum_{i=0}^2\binom{t+1}ip^{1-i}q^i\mu_p(\TT_i:[t+2,n])\nonumber\\
&\leq p^2+(t+1)p^3q+\binom{t+1}2p^5q^2\nonumber\\
&\leq p_0^2+(t+1)p_0^3q_0+\frac12(t+1)t\,p_0^5q_0^2\label{eq:h=2}\\
&=\frac3{2\sqrt t}+O(1/t).\nonumber
\end{align}
Thus \eqref{eq:h=2} is decreasing in $t$,
and $\mu_p(\HH)<\frac12p^t$ for $t\geq 10$.
\end{proof}

\begin{lemma}
If $h=3$ and $t\geq 70$ then $\mu_p(\HH)<\frac12 p^t$.
\end{lemma}
\begin{proof}
If $t\geq 70$ then $p_0\leq\frac1{8}$. 
For $0\leq i\leq h$, $\TT_i$ is 2-wise $(2i+1)$-intersecting 
by Lemma~\ref{lemmaMIFR}. 
As $p\leq p_0\leq\frac1{8}=\frac1{(2h+1)+1}\leq\frac1{(2i+1)+1}$,
it follows from (ii) of Lemma~\ref{EKR} that $\mu_p(\TT_i:[t+3,n])\leq p^{2i+1}$.
Then it follows from \eqref{eq:sum of G(Ti)} that
\begin{align}
\frac{\mu_p(\HH)}{p^t}&\leq
\sum_{i=0}^3\binom{t+2}ip^{2-i}q^i\mu_p(\TT_i:[t+3,n])\nonumber\\
&\leq \sum_{i=0}^3\binom{t+2}ip^{3+i}q^i
\leq \sum_{i=0}^3\binom{t+2}ip_0^{3+i}q_0^i\label{eq:h=3}\\
&=\frac16+O(1/{\sqrt t}).\nonumber
\end{align}
Then \eqref{eq:h=3} is decreasing in $t$, and less than 
$\frac12$ for $t\geq 4$.
\end{proof}

\begin{lemma}\label{lemma:h middle}
If $4\leq h\leq\frac{\sqrt t}2-\frac54$ and $t\geq 111$ then $\mu_p(\HH)<\frac12p^t$.
\end{lemma}

\begin{proof}
We need $t\geq 111$ because $4\leq \frac{\sqrt t}2-\frac54$.
Let $1\leq i\leq h$. It follows from Lemma~\ref{lemmaMIFR} that 
$\TT_i\subset 2^{[t+h,n]}$ is 2-wise $(2i+1)$-intersecting.
We claim that $p\leq\frac1{(2i+1)+1}$.
Indeed, using $h\leq\frac{\sqrt t}2-\frac54$, or equivalently,
$\frac2{\sqrt{4t}-1}\leq\frac1{2(h+1)}$, we have
\[
 p\leq p_0=\frac2{\sqrt{4t+9}-1}\leq\frac2{\sqrt{4t}-1}\leq
\frac1{2(h+1)}\leq\frac1{2(i+1)}.
\]
Thus, by (ii) of Lemma~\ref{EKR}, we have $\mu_p(\TT_i:[t+h,n])\leq p^{2i+1}$, 
and so by \eqref{eq:sum of G(Ti)}, 
\begin{align*}
\mu_p(\HH)&\leq\sum_{i=0}^h\binom {t+h-1}ip^{t+h-1-i}q^i\,p^{2i+1}
=p^t\,\sum_{i=0}^h\binom{t+h-1}ip^{h+i}q^i.
\end{align*}

We need to show that 
$\sum_{i=0}^h\binom{t+h-1}ip^{h+i}q^i<\frac12$.
As $pq$ is increasing in $p$ (for $p<\frac12$) it follows that
$p^{h+i}q^i=p^h(pq)^i\leq p_0^{h+i}q_0^i$, and
$\binom{t+h-1}ip^{h+i}q^i\leq\binom{t+h-1}ip_0^{h+i}q_0^i$.
We next claim that 
$\binom{t+h-1}ip_0^{h+i}q_0^i\leq\frac{27}{50}\binom{t+h-1}{i+1}p_0^{h+i+1}q_0^{i+1}$,
or equivalently, 
\begin{align}\label{p0q0/2}
 \frac{i+1}{t+h-i-1}\leq \frac{27}{50}\,p_0q_0. 
\end{align}
Since the LHS is maximized at $i=h$, it suffices to show that
$\frac{h+1}{t-1}\leq \frac{27}{50}p_0q_0$. For the LHS we use 
$\frac{h+1}{t-1}\leq\frac{\sqrt t-\frac12}{2(t-1)}\leq\frac1{2\sqrt t}$. 
Thus it suffices to show $1\leq\frac{27}{25}\sqrt t\,p_0q_0=:f(t)$. 
Noting that
\[
 f(t)=\frac{54 \sqrt{t} \left(\sqrt{4t+9}-3\right)}{25\left(\sqrt{4 t+9}-1\right)^2},
\]
we have
\[
\frac d{dt}f(t)=\frac{108 \left(t+3 \sqrt{4t+9}-9\right)}
{25\sqrt{t}\sqrt{4 t+9} \left(\sqrt{4t+9}-1\right)^3}>0,
\]
and $f(87)>1$. Thus we get \eqref{p0q0/2} for $t\geq 87$. We apply
$\binom{t+h-1}ip_0^{h+i}q_0^i\leq\frac{27}{50}\binom{t+h-1}{i+1}p_0^{h+i+1}q_0^{i+1}$
repeatedly, and we have
\begin{align*}
\sum_{i=0}^h\binom{t+h-1}ip_0^{h+i}q_0^i
&\leq\sum_{i=0}^h\left(\frac{27}{50}\right)^{h-i}\binom{t+h-1}hp_0^{2h}q_0^h\\
&=\frac{50}{23}\left(1-\left(\frac{27}{50}\right)^{h+1}\right)\binom{t+h-1}h(p_0^2\,q_0)^h\\
&<\frac{50}{23}\binom{t+h-1}h(p_0^2\,q_0)^h.
\end{align*}
Using $\binom{t+h-1}h<\left(\frac{e(t+h-1)}h\right)^h$ and
$h\geq 4$ it follows that
\[
\binom{t+h-1}h(p_0^2\,q_0)^h<
\left(e\left(1+\frac th\right)p_0^2\,q_0\right)^h\leq
\left(e\left(1+\frac t4\right)p_0^2\,q_0\right)^h\leq 0.69^h,
\]
where the last inequality holds for $t\geq 51$.
Consequently we have
\[
\sum_{i=0}^h\binom{t+h-1}ip^{h+i}q^i<\frac{50}{23}\cdot 0.69^h<\frac12
\]
for $h\geq 4$, and $\mu_p(\HH)<\frac12p^t$.
\end{proof}

\begin{lemma}\label{lemma:h big}
If $h\geq\frac{\sqrt t}2-\frac54$ and
$t\geq 241$ then $\mu_p(\HH)<\frac12p^t$.
\end{lemma}
\begin{proof}
Since $s\geq t+h\geq t+\frac{\sqrt t}2-\frac54$, $\HH$ is 2-wise 
$\lceil t+\frac{\sqrt t}2-\frac54\rceil$-intersecting. Then, by 
(i) of Lemma~\ref{EKR}, we have
\[
 \mu_p(\HH)\leq\left(\frac pq\right)^{t+\frac{\sqrt t}2-\frac54}.
\]
Since both $p/q$ and $1/q$ are increasing in $p$, it follows that
\[
\frac{\mu_p(\HH)}{p^t}\leq
\left(\frac pq\right)^{\frac{\sqrt t}2-\frac54}
\left(\frac 1q\right)^{t}
\leq
\left(\frac {p_0}{q_0}\right)^{\frac{\sqrt t}2-\frac54}
\left(\frac 1{q_0}\right)^{t},
\]
where $p_0$ is defined by \eqref{p0}, and $q_0=1-p_0$. The RHS of the above
inequality is
\[
 \exp\left(-(\log t-4)\frac{\sqrt t}4+O(1/t)\right)\to 0
\]
as $t\to\infty$, and in particular, $\mu_p(\HH)/p^t<1/2$ 
for $t\geq 241$.
\end{proof}

From Lemmas~\ref{lemma:h=1}--\ref{lemma:h big} we have
$\mu_p(\HH)/p^t<1/2$ for all $h\geq 1$ and $t\geq 241$. 
We also mention that the threshold for $t$ is increasing as $h$ increases.
This completes the proof of Proposition. 
\end{proof}

\section{Concluding remarks}
In this note we prove a stability result Theorem~\ref{thm5} for
3-wise $t$-intersecting families, which is valid for $t\geq 241$.
It seems that the same result holds for smaller $t$.
More strongly we conjecture the following.
Recall the definition of the family $\QQ$ from \eqref{eq:Q}.
\begin{conj}
Let $t\geq 4$ and $0<p\leq p_0(t)$.
If $\GG\subset 2^{[n]}$ is a $3$-wise $t$-intersecting family with
$\mu_p(\GG)>\max\{\mu_p(\QQ)\,\,\mu_p(\FF_1^t),\,\mu_p(\FF_2^t)\}$, 
then $\GG\subset\FF$ for some $\FF$ such that $\FF\cong\FF_0^t$.
\end{conj}

It would be interesting to extend Theorem~\ref{thm5} to the
$r$-wise $t$-intersecting families. To this end we need to generalize
Lemma~\ref{lemma:matome}, and we conjecture the following.
\begin{conj}
Let $\GG\subset 2^{[n]}$ be an $(r,t)$-maximal family, and let
$\HH$ be a shifted family obtained from $\GG$ by shifting.
If $\HH\subset\FF_i^t(r)$, then $\GG\cong\FF_i^t(r)$.
\end{conj}

For $i\geq 0$ define $p_i=p_i(t)$ by the minimum positive root of the equation 
$\mu_{x}(\FF_i^t)=\mu_{x}(\FF_{i+1}^t)$ in variable $x$.
It follows that
\[
p_i=\sqrt{\frac{i+1}t}+\frac{i+1}{2t}+O(t^{-\frac32}),\quad
\lim_{t\to 0}\frac{\mu_{p_i}(\FF_{i+2}^t)}{\mu_{p_i}(\FF_{i}^t)}=\frac{i+1}{i+2}.
\]
Finally we ask whether Theorem~\ref{thm5} can be extended to larger $p$ as follows.
For convenience let $\mu_p(\FF_{-1})=0$. 
\begin{prob}
Is the following statement true? Let $i\geq 1$. 
For every $\delta>0$ there exists $t_0$ such that
for all $t>t_0$ and $p$ with $p_{i-1}(t)\leq p\leq p_i(t)$, 
the following holds. If $\GG\subset 2^{[n]}$ is a $3$-wise 
$t$-intersecting family with
\[
 \mu_p(\GG)>(1+\delta)\max\left\{\frac{i-1}i\,\mu_p(\FF_{i-2}),\,
\frac{i+1}{i+2}\,\mu_p(\FF_{i+2})\right\},
\]
then $\GG\subset\FF$ for some $\FF\cong\FF_j^t$ where $j\in\{i-1,i,i+1\}$.
\end{prob}

\section*{Acknowledgment}
The author thanks the referees for their careful reading and many helpful suggestions,
in particular, for pointing out the error in the proof of Lemma~\ref{lemma:F1} in the
earlier version, and for correcting the error in the proof of Lemma~\ref{lemma:h middle}.
This research was supported by JSPS KAKENHI Grant No.~23K03201.

\end{document}